\documentclass{amsart}
\usepackage[T1]{fontenc}
\usepackage{lmodern}
\usepackage{amsmath}
\usepackage{amssymb}
\usepackage{amsthm}
\usepackage{graphicx}
\usepackage{hyperref}
\usepackage{booktabs}

\providecommand{\tightlist}{\setlength{\itemsep}{0pt}\setlength{\parskip}{0pt}}

\newtheorem{theorem}{Theorem}[section]
\newtheorem{proposition}[theorem]{Proposition}
\newtheorem{lemma}[theorem]{Lemma}
\newtheorem{conjecture}[theorem]{Conjecture}
\theoremstyle{definition}
\newtheorem{definition}{Definition}[section]
\theoremstyle{remark}
\newtheorem*{remark}{Remark}
\theoremstyle{plain}
\newtheorem{innercustomthm}{Theorem}
\newenvironment{customthm}[2][]{%
  \renewcommand\theinnercustomthm{#2}%
  \ifx\relax#1\relax\innercustomthm\else\innercustomthm[#1]\fi
}{\endinnercustomthm}
\newtheorem{innercustomprop}{Proposition}
\newenvironment{customprop}[2][]{%
  \renewcommand\theinnercustomprop{#2}%
  \ifx\relax#1\relax\innercustomprop\else\innercustomprop[#1]\fi
}{\endinnercustomprop}

\title[A deterministic \texorpdfstring{\(\sin^2\)}{sin2}-type algorithm for complex cubics]{A deterministic \texorpdfstring{\(\sin^2\)}{sin2}-type algorithm for complex cubic irrationalities with exact periodicity certificates}
\author{Ludovic Tagnon}
\address{Nancy, France}
\email{ludovic.tn.ry@gmail.com}
\subjclass[2020]{11J70, 11R16, 11Y65, 11H50}
\keywords{Hermite's problem, multidimensional continued fractions, sin2-algorithm, cubic fields, unit certificates}

\hypersetup{hidelinks,
  pdftitle={A deterministic sin2-type algorithm for complex cubic irrationalities with exact periodicity certificates},
  pdfauthor={Ludovic Tagnon},
  pdfsubject={Number theory: Hermite's problem for complex cubic irrationalities (Problem 4 of Karpenkov)},
  pdfkeywords={Hermite's problem, multidimensional continued fractions, sin2-algorithm, complex cubic fields, unit certificates, exact certification}}
\begin{document}

\begin{abstract}
Hermite asked in 1848 for a representation of real numbers whose eventual periodicity characterizes cubic irrationals. The totally real case was solved by Karpenkov's \(\sin^2\)-algorithm; the complex case (signature (1,1)) remains open --- Problem 4 of Karpenkov's Monatshefte paper, with the suggested analytic extension of the \(\sin^2\) formula. We study a deterministic algorithm implementing that suggestion: on (1,1) data the score expression is strictly negative (we prove a closed form), the algorithm selects the most negative score, and the frequent, provably exact score ties are resolved by a declared ordering. We report three groups of results with a strict separation of statuses. First, on a fixed sample of 205 complex cubic polynomials, every run of the sign-normalized campaign (195 seeds directly, 10 via the mirror generator \(-\alpha\)) closes projectively with an exact unit certificate \(\lambda=\zeta\cdot\varepsilon^k\in\mathcal{O}_K^\times\), and every transition is certified by exact score comparisons in \(\mathbb{Q}(\alpha)\): the computed dynamics is the exact algorithm's. An exhaustive campaign over the full box \([-3,3]^3\) (194 polynomials, no sampling freedom) closes 194/194 under the same certification, namely 97 positive-root representative expansions (82 distinct digit sequences) counted together with their mirror partners. Second, across 457 deformed starting bases the terminal cycle is always the same, by exact state-level comparison --- an observed invariant of the marked lattice \((p,\mathbb{Z}[\alpha])\), demonstrably not of the field alone --- and measured descent diagnostics quantify the reduction behaviour a proof would need. Third, for four fields we compute closed finite transition graphs; for the plastic field the full transition map is decided by exact comparisons in \(\mathbb{Q}(\alpha)\), and the plastic transition graph with its convergence to the terminal cycle is machine-checked in Lean 4, kernel-only, under the standard axioms. The admissible set of bounded height is not confined to the certified window (checked at a $+10$ widening, plastic field) and shows no sign of saturating. The general periodicity problem remains open. All data and scripts ship in a public archive with a portable verifier.
\end{abstract}

\maketitle

\section{Introduction}\label{sec:introduction}

\textbf{The problem.} By Lagrange's theorem, the continued fraction of a real number is eventually periodic if and only if the number is a quadratic irrationality. Hermite asked in 1848 \cite{He1850} whether an analogous representation exists for cubic irrationalities. Geometrically, periodicity is available: the Klein sail of a cubic field is periodic under the Dirichlet group of units \cite{GL08,Ts83}, the period operator being the integer matrix of multiplication by a unit. The algorithmic question --- to produce, from the number alone, an integer digit sequence whose eventual periodicity detects cubicity, the period revealing the unit --- is the open half. Karpenkov solved it for totally real cubic vectors with the \(\sin^2\)-algorithm and proved the corresponding periodicity theorem; for the complex signature (1,1), which includes all cubic roots such as \(\sqrt[3]{2}\) and \(\sqrt[3]{4}\), the problem is stated as open (Problem 4 of \cite{Ka24}, \S9: ``find a generalized Euclidean algorithm that is periodic for real cubic vectors in the non-totally-real case'', with the remark that ``the analytic extension of the formula for \(\sin^2\alpha\) to the complex numbers may be of use''), and with an elementary obstruction: a non-torsion unit of a complex cubic field acts on the complex place by a rotation of irrational angle (if a power had real complex embedding, separability and the absence of a quadratic subfield would force that power to be rational, hence \(\pm1\)); related non-closing phenomena for normalized approximations are studied in \cite{DH25}. Prior to this work, the empirical record for the \(\sin^2\)/HAPD family of algorithms on complex cubics consisted, to our knowledge (literature reviewed up to July 2026), of two published examples of Karpenkov; the Jacobi--Perron literature contains further periodic complex expansions (e.g.\ \cite{AR11}), and other multidimensional continued fraction families detect cubicity in the totally real case \cite{Da14}; but for no algorithm in any of these families is complex periodicity characterized.

Periodic algorithms that produce the fundamental unit in unit-rank-one fields go back to Vorono\"i and were generalized by Buchmann \cite{Bu85a,Bu85b} to all fields of unit rank one, including complex cubic fields, but these algorithms take the field or an order as input; here the input and question differ: the \(\sin^2\)-type algorithm runs on an arbitrary real vector, and periodicity of the output is the criterion characterizing cubicity, with the unit certificate recovered a posteriori and verified exactly in \(\mathbb{Q}(\alpha)\). Karpenkov's HAPD algorithm is observed to cycle on cubic data with complex conjugate roots \cite[Example~2.11]{Ka22}, while the complex case is explicitly separated in the conjectural discussion \cite[Conjecture~2 and Remark~2.12]{Ka22} and the rank-one Dirichlet-group structure appears in \cite[Example~4.2]{Ka22}; our contribution for that case is a deterministic \(\sin^2\)-analytic algorithm --- the extension suggested by Problem~4 of \cite{Ka24} --- together with a closed-form negativity lemma for the score, a declared exact tie-breaking rule, and instance-by-instance exact certification with a portable verifier, where the practical route discussed in \cite[\S2.4]{Ka22} detects periods from sufficiently good decimal approximations. Certification of periodicity by an algebraic certificate has a recent precedent in the repetend-matrix form of \v{R}ada--Starosta--Kala \cite{RSK24} for general multidimensional continued fractions, applied there to Jacobi--Perron-type algorithms; the certificates here are of a different nature --- fundamental-unit identities checked exactly in the field, together with exact score comparisons --- and cover the complex-conjugate cubic case.

\textbf{What this paper does.} We make Karpenkov's suggestion effective and study it at scale. Our contributions are fourfold. First, a \emph{determinization}: the minimization branch admits exact score ties (we exhibit them and certify some as algebraic equalities), so the algorithm is only well-defined once a tie-breaking convention is fixed; we adopt a lexicographic convention stated in Section 2. Second, \emph{mass periodicity with certificates}: on 205 complex cubic polynomials the deterministic algorithm is eventually periodic with an exact unit certificate (195 directly from the canonical seed, the remaining 10 for the sign-normalized mirror generator; Section 3), and we quantify the arithmetic of the certificates (the exponent \(k\), the closing identity); an exhaustive campaign over the full coefficient box \([-3,3]^3\) (\S3.3) removes the sampling freedom altogether. Third, a \emph{canonical-cycle phenomenon} and a measured \emph{reduction profile} (Sections 4--5): deformation experiments indicate that the terminal cycle is an invariant of the marked lattice --- robust under change of basis, though demonstrably not of the field alone (\S4.4) --- and the norm height along trajectories exhibits the quantitative behaviour that reduction-theory proofs require; we state these as calibrated measurements, not theorems. Fourth, \emph{machine-verified instances} (Section 6): for the field of \(x^3-x-1\) we verify by exhaustive exact computation a global convergence statement over an explicit finite set of admissible states, with three further complete instances (among them the field of \(\sqrt[3]{2}\)) and twelve orbit-level ones; we also document why the naive identification of this explicit set with ``all admissible states of bounded height'' fails --- the admissible set is wilder than the reachable one, which we regard as a structural finding in its own right.

\textbf{What this paper does not do.} It does not solve the complex case of Hermite's problem: no general periodicity theorem is proved here. It provides the deterministic algorithm, to our knowledge the largest certified empirical record to date, the measured shape of a possible proof, and the first certified finite-graph instances. Table~\ref{tab:status} (\S8.1) recapitulates, block by block, what is exact, what is measured, and what is conjectural; the methodology and the reproducibility archive are described in Section~9.

\textbf{Notation.} \(K=\mathbb{Q}(\alpha)\) a complex cubic field (signature (1,1)); \(\sigma_r\) the real embedding, \(\sigma_c\) one complex embedding; \(\varepsilon\) a fundamental unit, \(\operatorname{Reg}=\lvert\log\lvert\sigma_r(\varepsilon)\rvert\rvert\); \(N\) the absolute norm. States are triples of elements of a fixed lattice \(L\subset K\), acted on by unimodular transition matrices; the height of a state is \(H(s)=\max_i\lvert N(u_i)\rvert\) (the norm height, unit-invariant; \S5.1). \(B_K\) denotes the observed maximum of \(H\) over the computed canonical cycle (``band'') --- a computed quantity per field, not a hypothesized constant.

\section{The deterministic algorithm}\label{sec:algorithm}

\subsection{States and moves}\label{subsec:states-moves}

Fix a monic irreducible cubic \(p\in\mathbb{Z}[x]\) of signature (1,1), \(\alpha\) its real root, \(K=\mathbb{Q}(\alpha)\), and the lattice \(L=\mathbb{Z}[\alpha]\) with basis \((1,\alpha,\alpha^2)\). The \emph{canonical seed} of every run in this paper is that basis itself, \(s_0=(1,\alpha,\alpha^2)\) in identity coordinates, relabeled by the sorting convention below; all campaigns of \S3--\S5 start from \(s_0\).

\begin{definition}[state, admissibility]\label{def:state}
A \emph{state} is an ordered triple \(s=(u_1,u_2,u_3)\) of elements of \(L\), stored as integer coordinate vectors in the basis \((1,\alpha,\alpha^2)\); the absolute value \(d_0\) of the determinant of the coordinate matrix is invariant under the (unimodular) moves and is fixed by the seed. A state is \emph{admissible} if its three real embeddings are positive; before each step the triple is relabeled so that \(x=\sigma_r(u_1)> y=\sigma_r(u_2)> z=\sigma_r(u_3)>0\) (this is the sorting convention of the \(\sin^2\)-algorithm; the strictness is Lemma~\ref{lem:strict-order}). Implementation note: the mass-campaign engine sorts on 160-digit values with a stable sort, while the exact replays and the plastic-instance verifier decide the comparisons exactly.
\end{definition}

\begin{lemma}[strict real ordering]\label{lem:strict-order}
In any state whose coordinate determinant is nonzero, the three real embeddings are pairwise distinct. Consequently every admissible state sorts strictly as \(x>y>z>0\); for every integer \(0\le b\le\lfloor y/z\rfloor\) one has \(y/z-b>0\); and the move \(V(0,0,1)\) is always admissible before the isotropic exclusion of \S2.2. In particular the candidate set before that exclusion is never empty.
\end{lemma}

\begin{proof}
If \(\sigma_r(u_i)=\sigma_r(u_j)\), then \(\sigma_r(u_i-u_j)=0\). Since \(\sigma_r\) is an embedding of fields, it is injective, so \(u_i=u_j\), contradicting the nonzero determinant. Thus the sorted positive real embeddings are strict. If \(y/z=b\in\mathbb{Z}\), then \(\sigma_r(u_2-bu_3)=0\), hence \(u_2=bu_3\), again contradicting the determinant; this proves \(y/z-b>0\) throughout the stated range. Finally \(x/y>1\), so for \(a=b=0\) the bound for \(g\) is at least \(1\), and \(V(0,0,1)\) sends the real triple to \((x-y,y,z)\), which is positive.
\end{proof}

The \emph{moves} are the unimodular matrices of Karpenkov's algorithm, acting on the sorted basis: the family
\[
V(a,b,g)=
\begin{pmatrix}
1 & -g & gb-a\\
0 & 1 & -b\\
0 & 0 & 1
\end{pmatrix}
\]
for integers \(a,b,g\) in the admissible ranges \(0\le a\le \lfloor x/z\rfloor\), \(0\le b\le \lfloor y/z\rfloor\), and \(0\le g\le \lfloor (x/z-a)/(y/z-b)\rfloor\), excluding \((a,b,g)=(0,0,0)\) (the denominator is positive by Lemma~\ref{lem:strict-order}); and the single matrix
\[
W=
\begin{pmatrix}
1 & -1 & 0\\
0 & 1 & 0\\
-1 & 1 & 1
\end{pmatrix},
\]
admissible when \(z>x-y>0\).

These are exactly the candidate moves of the totally real \(\sin^2\)-algorithm \cite{Ka24}, transcribed without modification; our fidelity fixture is Example 1.21 of that paper (the algorithm as implemented reproduces its published expansion: pre-period 3, period 8). The \emph{candidate set} at a state is, by definition, the set of images of the state under the admissible moves above, minus the isotropic candidates (\(\langle n_1,n_1\rangle=0\), where the score of \S2.2 is undefined); the exclusion is part of the definition, not a runtime guard (\S2.2).

\subsection{The score and the minimization branch}\label{subsec:score}

For a candidate next state \(s'=(u'_1,u'_2,u'_3)\), write
\[
\xi=(\sigma_r(u'_1),\sigma_r(u'_2),\sigma_r(u'_3))\in\mathbb{R}^3,
\qquad
\nu=(\sigma_c(u'_1),\sigma_c(u'_2),\sigma_c(u'_3))\in\mathbb{C}^3.
\]
With \(\times\) and \(\langle\cdot,\cdot\rangle\) the formal (bilinear, non-Hermitian) cross and dot products on \(\mathbb{C}^3\), set \(n_1=\xi\times\nu\), \(n_2=\xi\times\bar{\nu}\), and

\[
\operatorname{score}(s') =
\operatorname{Re}\left[
\frac{\langle n_1\times n_2,n_1\times n_2\rangle}
{\langle n_1,n_1\rangle\langle n_2,n_2\rangle}
\right].
\]

In the totally real case (the second and third real embeddings in place of \(\nu\) and \(\bar{\nu}\)) this expression is the squared sine of the dihedral angle between the two planes they span with \(\xi\), the quantity whose \emph{maximization} defines the \(\sin^2\)-algorithm (Definition 1.16: the move of greatest \(\sin\alpha\)). On (1,1) embedding data the score has the following closed form.

\begin{lemma}[closed form and negativity of the score]\label{lem:closed-score}
Let \(s'\) be a candidate next state whose coordinate determinant is nonzero and for which \(\langle n_1,n_1\rangle\ne0\). Then

\[
\operatorname{score}(s') =
-\frac{|\det(\xi,\nu,\bar\nu)|^2\,|\xi|^2}
{|\langle n_1,n_1\rangle|^2}
\qquad(\langle n_1,n_1\rangle\ne 0).
\]
In particular the score is real and strictly negative.
\end{lemma}

\begin{proof}
Since \(\xi\) is real, \(n_2=\bar n_1\). The Lagrange identity \((a\times b)\times(a\times c)=a\cdot\det(a,b,c)\) (a polynomial identity, valid for the bilinear products over \(\mathbb{C}\)) gives \(n_1\times n_2=\xi\cdot\det(\xi,\nu,\bar{\nu})\), and \(\det(\xi,\nu,\bar{\nu})\) is purely imaginary because conjugation swaps \(\nu\) and \(\bar{\nu}\), hence negates it. The displayed formula follows. Finally, \(\det(\xi,\nu,\bar{\nu})\) equals the determinant of the coordinate matrix times the embedding determinant of the basis \((1,\alpha,\alpha^2)\), which is nonzero since \(p\) is irreducible; the denominator is nonzero by hypothesis.
\end{proof}

In the algorithm the coordinate determinant is nonzero and preserved by the unimodular moves. The bilinear quantity \(\langle n_1,n_1\rangle\) can in principle vanish on nonzero isotropic vectors; by definition such a candidate is inadmissible; the condition is decided exactly (\S2.3), and across all certified campaigns no isotropic candidate ever occurred and no candidate set ever became empty after this removal. (Totality is in fact settled a posteriori --- the candidate set never empties; see the remark after Conjecture~T*. The clause is kept inside T* only for self-containedness.) The selection convention, stated explicitly: the algorithm selects the candidate of \emph{most negative} score. Since \(\lvert\det(\xi,\nu,\bar{\nu})\rvert\) is constant across the candidates of any given step --- all the equivalence requires --- and indeed along the whole trajectory (the moves are unimodular and the engine stores exact integer states throughout, with no rescaling), this is equivalent to \emph{maximizing} the ratio \(\lvert\xi\rvert^2/\lvert\langle\xi\times\nu,\xi\times\nu\rangle\rvert^2\): the selected candidate is the one whose bilinear Pl\"ucker norm \(\lvert\langle\xi\times\nu,\xi\times\nu\rangle\rvert\) is smallest relative to the real-embedding size \(\lvert\xi\rvert\), at constant covolume. \emph{Why the most-negative branch?} Two reasons. Structurally, the uniform rule ``select the move of greatest \(\lvert\operatorname{score}\rvert\)'' reduces exactly to Definition~1.16 on totally real data (where the score is the nonnegative \(\sin^2\)) and to the most-negative branch on (1,1) data (where the score is negative); it is the continuation of the invariant \(\lvert\operatorname{score}\rvert\), not an ad hoc choice. Empirically, the opposite branch (least negative) fails outright: it systematically favours candidates of \emph{large} height (an empirical finding, consistent with the measured score--height law of \S5.3 through Lemma~\ref{lem:closed-score}), and an 8-step probe on the base cubics shows state heights reaching 10--20 decimal \emph{digits} with enumeration sizes up to \(1.1\cdot 10^5\) candidates per step and no return (the plastic field, whose canonical orbit has height 1 under the adopted rule, reaches 11-digit heights); the probe table ships in the archive. The most-negative branch is the only one of the two that behaves as a reduction.

This implements the analytic-extension suggestion of \cite[\S9, Problem~4]{Ka24}: selecting the most negative score --- equivalently the greatest \(\lvert\operatorname{score}\rvert\) --- is the sign-coherent continuation of the real-case maximization; it is also consistent in spirit with the Markov--Davenport-minimizing heuristic of the survey \cite{Ka22}.

\begin{proposition}[scores are exact]\label{prop:exact-scores}
Every candidate score lies in the real field \(\mathbb{Q}(\alpha)\), and every difference of candidate scores is decidable exactly. Indeed the complex-embedding data reduces to \(\mathbb{Q}(\alpha)\): writing \(\sigma_c(\alpha)=\operatorname{Re}\beta+J\) with \(\operatorname{Re}\beta=-(a_2+\alpha)/2\) and \(J^2=(\operatorname{Re}\beta)^2+a_0/\alpha\in\mathbb{Q}(\alpha)\) (for \(p=x^3-x-1\) this reads \(\sigma_c(\alpha)=-\alpha/2+J\), \(J^2=1-3\alpha^2/4\)), the score of any candidate is a rational expression in \(\sigma_c(\alpha)\) invariant under the conjugation \(J\mapsto -J\), hence an element of \(\mathbb{Q}(\alpha)\); and comparisons in \(\mathbb{Q}(\alpha)\) are decided exactly (\S2.3).
\end{proposition}

In practice scores are evaluated at 160 decimal digits, and the resulting selections have been re-certified exactly for the mass and box campaigns and for the plastic closure (\S2.3; the three other closure tables are 160-digit computations, \S6.3, and the auxiliary campaigns are inventoried in \S8.3); the certificates of \S2.4 are independent of working precision.

\subsection{Ties and the deterministic convention}\label{subsec:ties}

The selection rule alone does not define an algorithm: exact score ties occur. Mathematically, a \emph{tie} is an exact equality of candidate scores --- decidable, since scores lie in \(\mathbb{Q}(\alpha)\) (Proposition~\ref{prop:exact-scores}); the algorithm is defined with exact comparisons throughout. (Implementation note: the mass campaign evaluates scores at 160 decimal digits and routes relative gaps below \(10^{-60}\) to the tie-handling branch; all 2265 transition steps of the 205 mass-campaign trajectories, and all 2285 transitions of the \S6 plastic graph, were replayed with exact comparisons and confirmed without exception; the replay recomputes the full candidate set --- floor bounds, \(W\)-move admissibility signs, isotropy checks --- not merely the score comparisons, so the certification covers enumeration and selection alike. The threshold is a trigger, not part of the definition.) A concrete example from the plastic closure: at the state with coordinate rows \((-1,-1,2),(-1,1,1),(0,-1,1)\), the three candidates \(V(0,1,2)\), \(V(2,2,2)\), \(V(2,2,3)\) have exactly equal scores in \(\mathbb{Q}(\alpha)\), and the ordering selects \(V(0,1,2)\). Remark 1.18 (arXiv version of \cite{Ka24}) already proposes determinizing by an ordering: standard lexicographic on the \(V\) moves, together with a rank for \(W\). Our convention, fixed in advance of all experiments reported here, follows that suggestion --- tied \(V\) moves are ordered lexicographically by \((a,b,g)\) --- with one variation: a tied \(W\) move is given priority over tied \(V\) moves (the reverse of the rank suggested there). The variation is immaterial on all our data: no observed minimum tie requiring the tie-break, across all campaigns, involves \(W\) (see the archive tie logs for the per-campaign tallies). Our contribution on ties is not the convention but their exact algebraic certification (below). On the 17-cubic determinization sample (the seven cubics of the initial complex campaign plus ten enumerated additions; archive), tie steps occur in \(x^3+x-1\) (2 steps) and \(x^3+2x-1\) (1 step); on the tie-bearing polynomial \(x^3+x-1\) the digit output was checked bit-identical at 120 and 240 digits of working precision.

Ties are not a numerical artifact, and for the instance of Section 6 the floating comparison has been eliminated altogether. Since the score of every candidate is an element of \(\mathbb{Q}(\alpha)\) (\S2.2), score comparisons are decidable exactly. Re-deciding all 2285 transitions of the Section 6 state set with certified dyadic-interval arithmetic in \(\mathbb{Q}(\alpha)\) --- declaring a tie only upon exact vanishing of the score difference --- confirms every recorded digit and every recorded target state (0 discrepancies, 2,974,862 candidates enumerated), and reveals that exact algebraic ties of the minimum occur at 156 of the 2285 states. Two consequences: the transition map of Section 6 depends on no floating-point threshold; and the tie-breaking convention is a \emph{constitutive} part of the algorithm --- on this state set, exact minimum ties are common (6.8\% of states), so the minimization rule alone genuinely underdetermines the dynamics. The convention is declared, not intrinsic: nothing in the score singles it out, and OP5 (\S8.2) asks for a classification of the exact ties that would either ground it structurally or replace it. Working precision of 160 digits sufficed to separate every non-tied comparison, a comfortable certification margin.

\subsection{Periodicity certificates}\label{subsec:periodicity-certificates}

A run is declared \emph{periodic} only upon an exact certificate. If the canonical forms of the states (Definition~\ref{def:canonical-form}, \S6.1) at steps \(m\) and \(m+P\) coincide projectively, the scalar \(\lambda\in K\) with \(s_{m+P}=\lambda\cdot s_m\) (componentwise, in exact field arithmetic) is computed and four checks are performed in PARI: the same \(\lambda\) relates all three components; \(\lvert N_{K/\mathbb{Q}}(\lambda)\rvert=1\); \(\mathbb{Q}(\lambda)=K\); and \(\lambda\in\mathcal{O}_K^\times\) --- verified constructively by computing its bnfisunit decomposition \(\lambda=\zeta\cdot\varepsilon^k\), which succeeds on all 205 certificates below. (Both states lie in the lattice \(L\) by construction --- the engine stores exact integer coordinates at every step and never rescales --- so no assumption that multiplication by \(\lambda\) preserves \(L\) is needed; the certified statement is a return within \(L\) up to a unit of \(\mathcal{O}_K\).) Only then is the pair (pre-period \(m\), period \(P\)) recorded, together with \(\lambda\) (indexing convention: states are indexed from the seed \(s_0\), as in \S2.1; the recorded pre-period \(m\) is the index of the first state whose canonical form recurs, clamped below at 1 --- a run whose seed already lies on the terminal cycle records \(m=1\), as does a run with exactly one transient state --- and the certificate identity \(s_{m+P}=\lambda\cdot s_m\) holds at the recorded \(m\)). The certificate is therefore not a numerical near-return: it exhibits the unit. Its dynamical consequence deserves to be stated explicitly: a projective return propagates, so the certificate implies periodicity of the digit sequence itself.

\begin{lemma}[projective return implies digit periodicity]\label{lem:equivariance}
Let \(s\) be an admissible state and \(\lambda\in K^\times\) with \(\sigma_r(\lambda)>0\). Then \(s\) and \(\lambda\cdot s\) have the same sorting permutation, the same admissible-move ranges, the same \(W\)-move admissibility, and the same isotropic exclusions; and since the moves are linear (\(M\cdot(\lambda s)=\lambda\cdot(Ms)\)), corresponding candidates have equal scores, so the selected digit at \(\lambda\cdot s\) equals the selected digit at \(s\) and the successor states again differ by the factor \(\lambda\). Consequently a certified return \(s_{m+P}=\lambda\cdot s_m\) with \(\sigma_r(\lambda)>0\) gives \(s_{t+P}=\lambda\cdot s_t\) and \(\mathrm{digit}_{t+P}=\mathrm{digit}_t\) for all \(t\ge m\): the digit sequence is eventually periodic, with period dividing \(P\).
\end{lemma}

\begin{proof}
The sorting comparisons are strict by Lemma~\ref{lem:strict-order}, and the enumeration bounds \(\lfloor x/z\rfloor\), \(\lfloor y/z\rfloor\), \(\lfloor(x/z-a)/(y/z-b)\rfloor\) and the \(W\)-move condition \(z>x-y>0\) depend only on ratios of the positive reals \(\sigma_r(u_i)\), unchanged when each \(u_i\) is multiplied by \(\lambda\) with \(\sigma_r(\lambda)>0\). The score is homogeneous of degree zero: under \(u\mapsto\lambda u\), both \(\lvert\det(\xi,\nu,\bar\nu)\rvert^2\,\lvert\xi\rvert^2\) and \(\lvert\langle n_1,n_1\rangle\rvert^2\) scale by \(\sigma_r(\lambda)^4\lvert\sigma_c(\lambda)\rvert^4\), so Lemma~\ref{lem:closed-score} is invariant; isotropy is preserved for the same reason. The tie ordering is defined on move labels, hence state-independent. Finally, for a certificate \(\sigma_r(\lambda)>0\) holds automatically: \(\lambda\) is a ratio of elements with positive real embeddings.
\end{proof}

\subsection{The algorithm in one block}\label{subsec:pseudocode}
For reference, one full step of the deterministic algorithm:

\begin{quote}\ttfamily\small
Input: state $s=(u_1,u_2,u_3)$, integer vectors in the basis $(1,\alpha,\alpha^2)$.\\
1. Sort $s$ so that $\sigma_r(u_1)>\sigma_r(u_2)>\sigma_r(u_3)$; require all three positive.\\
2. Enumerate the admissible moves: $V(a,b,g)$ over the integer ranges of \S2.1, and $W$ if $z>x-y>0$.\\
3. For each candidate state, discard it if $\langle n_1,n_1\rangle=0$; otherwise compute its score (\S2.2).\\
4. Select the candidate of most negative score; on an exact tie, apply the declared ordering (\S2.3).\\
5. Output the digit ($V(a,b,g)$ or $W$) and the new state; the permutation of step 1 is not part of the digit.\\
Periodicity is detected on canonical forms (\S2.4), and every selected transition of the mass and box campaigns, and of the plastic closure, has been replayed with exact arithmetic (the other closure tables are computed at 160 digits, \S6.3): the candidate ranges (floor bounds), the $W$-move admissibility signs, the isotropy checks, and the score comparisons are all decided exactly in $\mathbb{Q}(\alpha)$.
\end{quote}

\section{Mass periodicity with exact unit certificates}\label{sec:mass}

\subsection{The sample}\label{subsec:sample}

The sample is fixed in advance and deterministic: the first 200 monic cubics \(x^3+a_2x^2+a_1x+a_0\), coefficients enumerated lexicographically by \((a_2,a_1,a_0)\) ascending from \((-6,-6,-6)\) within \([-6,6]^3\), restricted to irreducible polynomials of signature (1,1); plus five literature-tagged polynomials: \(x^3-2\), \(x^3-17\) and \(x^3-3x^2-2\) (the Jacobi--Perron unit computations of Adam--Rhin \cite{AR11}) and \(x^3-4\), \(x^3+2x^2+x+4\) (the two complex examples of Karpenkov; the latter has discriminant \(-416\)). Total: 205 polynomials. Distinct polynomials may generate isomorphic fields (they do: the 205 polynomials span 180 distinct fields, \S4.4); certificates are computed per polynomial. An exhaustive complementary campaign, with no sampling freedom at all, is reported in \S3.3.

\subsection{Result: 205/205 with certificates}\label{subsec:mass-result}

\begin{customthm}[certified finite computation]{A}\label{thm:A}
On the 205 listed polynomials, the deterministic algorithm (step cap 500, working precision 160 digits) yields: 195 periodic with valid certificate, and 10 runs halted with the diagnostic ``no positive third coordinate'' --- precisely the polynomials whose real root is negative, for which the sorting convention \(z>0\) cannot be met along the seed orbit. For those 10 we run the monic mirror \(q(x)=-p(-x)\) (same field, same discriminant, generator \(-\alpha\) with positive real root): all 10 are periodic with valid certificate. Thus: for 195 polynomials the stated seed itself yields a certified periodic expansion, and for the remaining 10 the sign-normalized mirror generator does (same field); the sign-normalized 205-input campaign succeeds \textbf{205/205}. Periods range over \([1,50]\) with median 8; pre-periods over \([1,5]\). No polynomial required more than the step cap, and no certificate check failed. Status, stated precisely: \emph{every} decision of \emph{every} run is exact-certified --- all 2265 transition steps of the 205 recorded trajectories were replayed with exact score comparisons in \(\mathbb{Q}(\alpha)\) (the field-wise closed form of \S2.2, re-derived and numerically verified per field), confirming the selected digit at every step with zero discrepancies. The 14 exact score ties encountered en route are algebraic equalities, all resolved to the recorded digit by the declared ordering, and 160 working digits sufficed to separate every non-tied comparison. The trajectories are therefore those of the exact algorithm, and the terminal unit certificates apply to it.
\end{customthm}

Theorem~\ref{thm:A} is the mass certificate statement used below.

\subsection{An exhaustive box campaign}\label{subsec:box}

The lexicographic sample of \S3.1, while fixed in advance, is concentrated: its 200 enumerated polynomials all have \(a_2\in\{-6,-5,-4\}\). To remove any sampling freedom we ran a complementary campaign on an \emph{exhaustive} family: all monic cubics with coefficients \((a_2,a_1,a_0)\in[-3,3]^3\), irreducible of signature (1,1) --- 194 polynomials, determined by the box alone. The box is closed under the mirror \(p(x)\mapsto -p(-x)\), i.e.\ \((a_2,a_1,a_0)\mapsto(-a_2,a_1,-a_0)\), and contains no self-mirror polynomial (a self-mirror would have \(a_2=a_0=0\), hence be divisible by \(x\)), so its 194 polynomials form 97 mirror pairs, each pair with exactly one positive-real-root member. The overlap with the 205-polynomial sample is two polynomials (\(x^3-2\) and \(x^3-3x^2-2\); their box runs reproduce the archived \S3.2 rows identically, and served as fixtures).

\begin{customthm}[certified finite computation]{A\('\)}\label{thm:Aprime}
On the 194 polynomials of the box, the sign-normalized campaign (same conventions and caps as Theorem~A) succeeds \textbf{194/194}: the 97 positive-root members close directly from the canonical seed, and each of the 97 negative-root members halts by sign and closes through its mirror --- the other member of its pair, with the identical digit sequence. The campaign thus certifies 97 representative expansions (82 distinct digit sequences) --- 420 distinct transition steps; counted over all 194 runs (each expansion reached from both members of its pair), the archive's exact replay covers 840 steps. Every transition of every one of the 194 recorded runs was replayed with exact comparisons in \(\mathbb{Q}(\alpha)\) --- enumeration bounds, \(W\)-move admissibility, isotropy and score comparisons --- with zero discrepancies. The 25 tie steps among the 420 distinct transitions (50 counted over all 194 runs) are exact algebraic equalities, all resolved to the recorded digit by the declared ordering. Certificate exponents lie in \(k\in[1,3]\), periods in \([1,10]\), pre-periods in \([1,4]\) --- within the ranges of Theorem~A.
\end{customthm}

Two readings. First, the sampling objection is closed: the box leaves no enumeration choice, and the phenomenon persists unchanged. Second, the mirror structure of the box makes the sign-normalization of Theorem~A self-contained: on a mirror-closed family, every halt-by-sign is resolved \emph{within} the family, and the halting diagnostic is exactly the sign of the real root, 97/97.

\subsection[The certificate unit and the closing identity]{The certificate unit: \texorpdfstring{$\lambda = \zeta\cdot\varepsilon^k$}{lambda = zeta epsilon^k} and the closing identity}\label{subsec:unit}

Writing each certificate as \(\lambda=\zeta\cdot\varepsilon^k\) with \(\varepsilon\) the fundamental unit (contracting in \(\sigma_r\), \(\operatorname{Reg}=-\log\lvert\sigma_r(\varepsilon)\rvert>0\)) and \(\zeta\in\{\pm 1\}\) the torsion part, the exponent distribution over the 205 runs is

\begin{center}
\begin{tabular}{rrrrrrr}
\toprule
\(k=1\) & \(k=2\) & \(k=3\) & \(k=4\) & \(k=6\) & \(k=12\) & \(k=14\)\\
\midrule
169 & 14 & 11 & 5 & 4 & 1 & 1\\
\bottomrule
\end{tabular}
\end{center}

with \(\zeta=+1\) on 109 and \(\zeta=-1\) on 96 runs. The certificate unit is thus fundamental (up to sign) in 82\% of the sample but not always: the extremes are \(x^3-5x^2-x-3\) (disc \(-1984\), period 8, \(k=14\)) and \(x^3-4x^2-3x-5\) (disc \(-2783\), period 3, \(k=12\)). The exponent is not an artifact of the search: after a certificate is found, the following closing identity gives an independent end-to-end check. Let \(\bar g\) be the mean per-step logarithmic gain of the real embedding over one period and \(P\) the period length. Then

\[
P\cdot \bar g=k\cdot\operatorname{Reg}
\]
holds with relative residual below \(10^{-78}\) on all 205 runs (recomputation at high precision, archived script; the regenerated tables of the archive carry the audited regulator column). The identity is essentially telescopic once the certificate holds; we report it as an end-to-end consistency check of the numerical chain. The geometric reading: over one period the algorithm contracts the real embedding by exactly the amount that multiplication by \(\zeta\cdot\varepsilon^k\) contracts it --- the period ``spends'' \(k\) regulators.

Figure~\ref{fig:kdist} records this distribution.

\begin{figure}[t]
\centering
\includegraphics[width=0.62\linewidth]{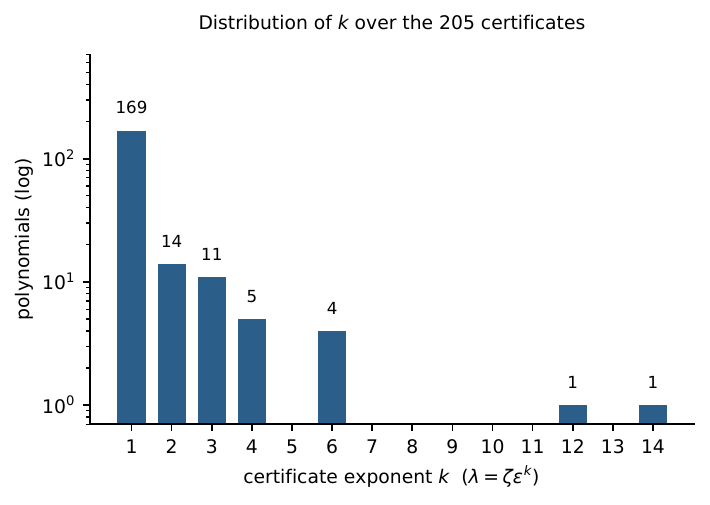}
\caption{Distribution of the certificate exponent \(k\) over the 205 certificates. The unit is fundamental (up to sign) in 82\% of the sample; the extremes \(k=12,14\) occur at non-maximal generators of fields that also carry \(k=1\) (\S4.4).}
\label{fig:kdist}
\end{figure}

\subsection{The rotation at the complex place}\label{subsec:rotation}

A basic obstruction for complex Hermite expansions is that a non-torsion certificate unit acts on the complex embedding by rotation of angle \(\theta(\lambda)\) with \(\theta/2\pi\) irrational: if \(\theta/2\pi\) were rational, a nonzero power of \(\lambda\) would have real complex embedding and hence would be rational, forcing \(\lambda\) to be torsion. Related complex-place non-closing phenomena are studied in \cite{DH25}. Our data is consistent: the measured \(\theta/2\pi\) shows no small-denominator rational within the detection bound on the tabulated runs (per-run diagnostic in the archive), and no correlation is detected between the period data and the Diophantine quality of \(\theta/2\pi\) (Spearman \(\rho=0.06\), \(p=0.38\), \(n=195\) --- the direct runs, mirrors excluded from this diagnostic; threshold \(\alpha=0.01\), fixed in advance; an exploratory diagnostic on a deterministic sample, carrying no inferential weight, \S5). Periodicity, when it occurs, is carried by the full state --- not by a closing-up at the complex place. This is a descriptive statement about our sample, not a theorem.

\subsection{Agreement with the published record}\label{subsec:published-record}

Where the literature provides comparison points, the certificates agree. The units of Adam--Rhin for \(x^3-2\) (disc \(-108\)) and \(x^3-3x^2-2\) (disc \(-324\)) are recovered with \(k=1\) (certificate unit equal to the published unit up to sign and inversion). For \(x^3-4\), the deterministic algorithm yields pre-period 1, period 4; Karpenkov's published HAPD computation of the same example has pre-period 6, period 4 --- same period, different pre-period, as expected of different algorithms. For disc \(-416\) (via its mirror), the \(\sin^2\)-expansion is (1, 4) against the published Jacobi--Perron (6, 2); the certificate unit is a unit of the same field in both readings. Different algorithms produce different expansions; the field-level invariant they must share --- the unit group --- is shared.

\subsection{Relation to sail periodicity}\label{subsec:sail-periodicity}

Geometric periodicity for cubic fields is classical: the Klein sail \cite{Kl1895} is periodic under the Dirichlet group, the period operator being multiplication by a unit. What the certificates above recover of that picture is exactly the unit: \(\lambda\in\mathcal{O}_K^\times\) is the scalar by which the state returns --- a period operator realized dynamically by the algorithm. What they do not establish is the combinatorial correspondence between the digit cycle and the sail structure (vertex set, fundamental domain) of the same field --- we have not compared the two, and regard that comparison as a natural next step (\S8).

\section{The canonical-cycle phenomenon}\label{sec:canonical-cycle}

\subsection{Deformation protocol}\label{subsec:deformation}

The seed state of Section 3 is one basis of \(L\) among infinitely many. To probe whether the terminal cycle depends on the starting basis, we deform: \(G\in\mathrm{SL}_3(\mathbb{Z})\) is applied to the seed (products of elementary matrices from seeded generators; protocol, seeds and acceptance rates in the archive), the deformed triple is run if admissible (positive sorted embeddings), and the terminal cycle is compared to the canonical one by \emph{exact equality of canonical state sequences modulo rotation}: the projective states along the two cycles (exact field-element ratios, the same canonicalization used for periodicity detection) must coincide as cyclic sequences. All 457 deformed runs reported below pass this state-level test --- each re-executed deterministically from the archived deformation matrix, reproducing its archived status, pre-period, period and certificate (457/457; comparison table in the archive). The digit cycle up to rotation and the multiset of state heights along the cycle, the runtime criterion of the original campaigns, are retained as secondary diagnostics; on no run do the two criteria disagree. On 30 fields (the Section 3 sub-sample with smallest discriminants): a first campaign of 300 seeded deformations yielded 71 admissible starts, 71/71 periodic, 71/71 on the canonical cycle; a second campaign forced coverage (10 admissible deformations per field, 1210 attempts): 300/300 periodic, 300/300 on the canonical cycle.

\subsection{Stress test at large heights}\label{subsec:stress}

A third campaign pushed the deformation heights: 90 runs with initial height \(H_0\) up to \(3.4\cdot 10^9\) times the cycle height (log-ratio up to 21.95). Of these, 86 completed: 86/86 periodic, 86/86 on the canonical cycle, certificates valid. The remaining 4 hit the stress-campaign enumeration guard (more than \(10^6\) candidate moves to enumerate at a single step) and were stopped without result --- an honest cost boundary of the method at extreme heights, not a counterexample; we quantify this enumeration hardness further in \S6.4.

Figure~\ref{fig:descent} displays the completed stress runs.

\begin{figure}[t]
\centering
\includegraphics[width=0.78\linewidth]{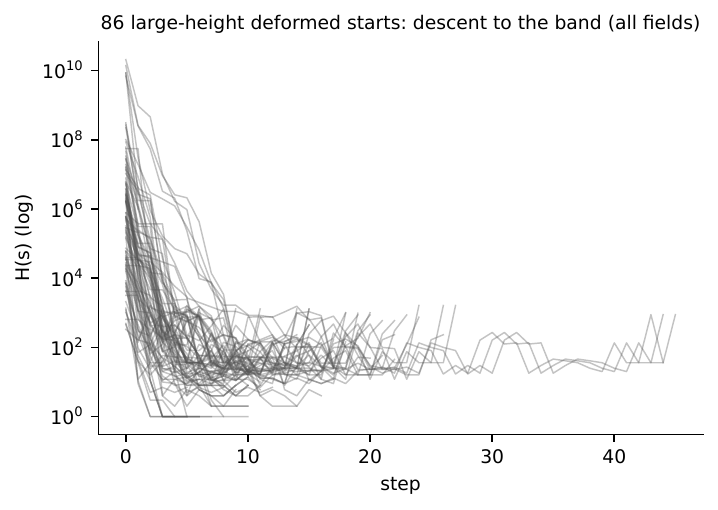}
\caption{The 86 completed large-height stress runs (all 30 fields, log scale): descent from seed heights up to \(3.4\cdot 10^9\) into the band within a few steps, then bounded oscillation. The two below-band starts rise exactly to \(B_K\) (\S5.1).}
\label{fig:descent}
\end{figure}

\subsection{Orbit-level closures}\label{subsec:orbit-closures}

Independently of deformations, for 12 fields (the 8 smallest of the determinization sample plus 4 quantile-selected by band height) the full forward closure of the canonical orbit was computed (\S6 method). Since the forward closure of a single orbit of a deterministic map has exactly one terminal cycle by construction, these runs are \emph{not} evidence for uniqueness of the attractor; their value is that the twelve canonical trajectories are verified exactly, and that they supply the per-field cycle data used in \S5.1. Evidence for uniqueness comes from the four \emph{complete} closures of \S6, whose seed sets contain all in-window admissible states and all deformation states. (Terminology: a \emph{complete closure instance} closes the full seed set; \emph{orbit-level} closes the canonical orbit alone.)

\subsection{The conjecture}\label{subsec:conjecture}

The cumulative record is 457/457 deformed starts reaching the canonical cycle, together with the four complete closures of \S6 (single cycle in each full state space). We state the phenomenon as a conjecture. What the data supports is invariance for the \emph{algorithm as fixed here} --- polynomial \(p\), lattice \(\mathbb{Z}[\alpha]\), sorting and tie-breaking conventions included:

\begin{conjecture}[T*: periodicity]\label{conj:Tstar}
For each monic complex cubic \(p\) (with \(L=\mathbb{Z}[\alpha]\) and the conventions of \S2), every admissible trajectory is well-defined indefinitely (in particular, the post-isotropy candidate set never becomes empty) and is eventually periodic.
\end{conjecture}

\begin{remark}[status of the well-definedness clause]
The nonemptiness clause of T* is in fact settled: Lemma~\ref{lem:strict-order}
always supplies the admissible candidate \(V(0,0,1)\) before isotropic
removal, and the companion theory paper proves that on \(\mathbb{Q}\)-basis
data no isotropic candidate ever arises \cite{Tag26b}, so the post-isotropy
candidate set is never empty. The clause
is retained in T* only to keep the periodicity statement self-contained; the
genuinely open content is the eventual periodicity.
\end{remark}

\begin{conjecture}[C*: canonical cycle]\label{conj:Cstar}
Moreover, all admissible trajectories reach the \emph{same} terminal cycle, an invariant of \((p,L,\mathrm{convention})\).
\end{conjecture}

Our own sample decides the field-level question, negatively. Grouping the 205 polynomials by field (polredabs): 180 distinct fields, 19 of them represented by several polynomials --- and in 18 of the 19 groups the terminal data \((k,\zeta,P)\) differs across generators. The field of \(\sqrt[3]{2}\) appears at \(x^3-2\), \(x^3-4\), \(x^3-6x^2-4\) and \(x^3-6x^2+6x-2\) with \(k=1,2,6,1\); the field of \(x^3+x-1\) carries both \(k=1\) and the sample maximum \(k=14\) at different polynomials. Even sharing both the field and the polynomial discriminant does not force agreement (5 such pairs disagree). The invariance supported by the data is thus exactly at the level stated in C*: robust under \(\mathrm{SL}_3(\mathbb{Z})\) change of basis of the embedded lattice (457/457, \S4.1--4.2), not under change of generator or embedding (OP8 asks for the right functorial description). C* is stronger than periodicity (it adds uniqueness of the attractor). Section 6 shows that, for a given field, the conjunction ``periodicity + canonical cycle over an explicit state set'' is \emph{decidable by finite computation}, and settles it for explicit sets in 12 fields (four of them by complete closure). No general proof is claimed.

\section{A measured reduction profile}\label{sec:reduction-profile}

A periodicity proof in the style of the totally real case needs a reduction theory: a height that eventually descends, a compact band where the dynamics is trapped, and finiteness of states in the band. This section reports the \emph{measured} profile of these three ingredients --- calibrated constants, sample sizes and exceptions stated exactly. These are measurements on samples fixed in advance, not theorems; we regard them as the quantitative specification a proof would have to reproduce. All quoted correlations are descriptive: the samples are deterministic and trajectories are not independent, so no inferential weight is claimed.

\subsection{The height and the band}\label{subsec:height-band}

For \(u\in L\setminus\{0\}\), \(H(u)=\lvert N_{K/\mathbb{Q}}(u)\rvert=\lvert\sigma_r(u)\rvert\cdot\lvert\sigma_c(u)\rvert^2\in\mathbb{Z}_{\ge 1}\); for a state, \(H(s)=\max_i H(u_i)\). We use \(H\) purely as a \emph{norm height}: it is unit-invariant, and on (1,1) embedding data it agrees with the Markov--Davenport form evaluated at \(u\) (no further use is made of that identification). The \emph{band} \(B_K\) is the maximum of \(H\) over the canonical cycle. Across the 205-polynomial sample, \(\log B_K\) fits \(1.372\cdot\log\lvert\mathrm{disc}\rvert-5.20\) with \(R^2=0.917\); on the 12 closure fields, the canonical excursion ceiling \(X^{*}_{\mathrm{can}}\) (maximum of \(H\) over the whole canonical orbit including pre-period) correlates with \(\log\lvert\mathrm{disc}\rvert\) at \(\rho=0.94\) (\(R^2=0.89\)). In 11 of these 12 fields \(X^{*}_{\mathrm{can}}\) equals \(B_K\) exactly: the canonical orbit never exceeds its terminal band. Deformed orbits can exceed it transiently, but the excess tracks the \emph{seed} height and decays along the run --- in the large-height stress runs (\S4.2), the running maximum after the start never exceeded \(\max(H_0,B_K)\) (86/86; for the 84 starts above the band, never exceeded \(H_0\), and the two below-band starts rose exactly to \(B_K\), as reaching the cycle requires); the phenomenon ``excursion'' is a property of perturbed starts, not of the field.

Figure~\ref{fig:band} shows this fit.

\begin{figure}[t]
\centering
\includegraphics[width=0.62\linewidth]{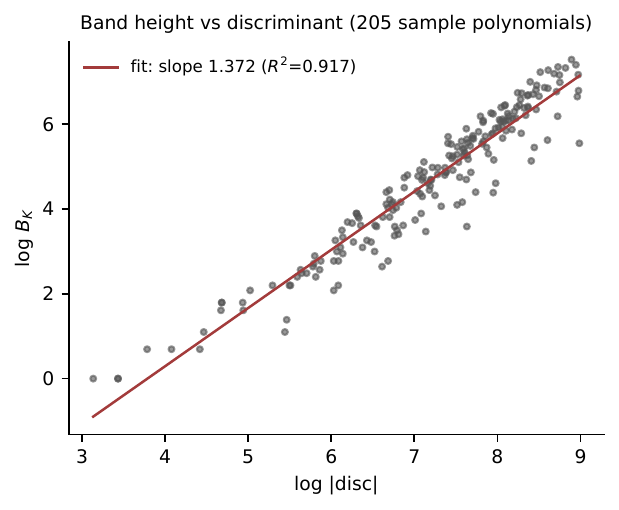}
\caption{Band height \(B_K\) against the polynomial discriminant for the 205 sample polynomials (log-log): slope 1.372, \(R^2=0.917\), recomputed from the archived tables by the figure script.}
\label{fig:band}
\end{figure}

\subsection{Four descent diagnostics}\label{subsec:diagnostics}

On the 90 height-diagnostic runs (30 canonical + 60 moderate deformations; 1016 steps, distinct from the large-height stress set of \S4.2), with candidate heights computed at every step --- exhaustively at 1007 of the 1016 steps (up to the \(5\cdot 10^4\) height-diagnostic exhaustiveness cap), and at the 9 larger steps (up to \(1.39\cdot 10^6\) candidates) by a fixed seeded sample of 1000 top-score plus 1000 random candidates:

\begin{enumerate}
\def\labelenumi{\arabic{enumi}.}
\tightlist
\item
  \textbf{Existence below the envelope (empirical Minkowski diagnostic, \(C_1=1\)).} At every one of the 1016 steps there exists an admissible candidate with \(H\le\max(H(s),B_K)\). The measured constant is exactly 1.
\item
  \textbf{Chosen-step envelope (\(C_2=11\), attained).} The candidate actually selected by the score satisfies \(H(\mathrm{selected})\le 11\cdot\max(H(s),B_K)\) at all 1016 steps --- with equality (11.0 exactly) attained at one step, so the constant has no slack on this sample --- and \(\le 1\cdot\max(H(s),B_K)\) on 99.6\% of them (median ratio 0.12).
\item
  \textbf{Pure descent above the band.} In the separate 90-run large-height stress set of \S4.2, the 84 runs started above the band never rose above their starting value (constant exactly 1); the two below-band starts rose exactly to \(B_K\). Return below \(H_0\) took at most 2 steps in 87 of those 90 stress runs (partial trajectories of the 4 guard-stopped runs included); per-run contraction: descriptive Spearman \(\rho=0.669\) on the deterministic sample.
\item
  \textbf{No re-ascent after reaching the band (L2 check).} On 794 steps of the earlier deformation campaigns: once a trajectory's height enters \([1,B_K]\), it never again exceeds \(11\cdot B_K\). Exceptions: 0/794.
\end{enumerate}

\subsection{The score--height frontier}\label{subsec:frontier}

The link a proof must control is between the \emph{score} (what the algorithm minimizes) and the \emph{height} (what a reduction theory needs to control). We measured it directly, twice. First, on 70 sampled steps across 27 fields, every candidate in the near-minimal score window \(\operatorname{score}\le \operatorname{score}_{\min}+\delta\cdot\lvert\operatorname{score}_{\min}\rvert\) (three widths, fixed in advance, \(\delta\in\{10^{-6},10^{-3},10^{-2}\}\)) satisfies \(H\le 11\cdot\max(H(s),B_K)\). Exceptions: 0/70 windows. Second, a dedicated anatomy campaign (30 fields, 808 instrumented steps, 973,914 candidates, exhaustive at all but 8 steps) recorded, for \emph{every} candidate, the pair \((r,\delta_{\mathrm{rel}})\) with \(r=\log(H_{\mathrm{cand}}/\max(H(s),B_K))\) and \(\delta_{\mathrm{rel}}\) the relative score gap to the minimum. Findings: (i) zero candidates with \(r>\log 11\) and \(\delta_{\mathrm{rel}}<10^{-3}\) --- the frontier holds on a candidate sample 8.6 times larger (973,914 against 113,114 candidates); (ii) the \emph{floor} of \(\delta_{\mathrm{rel}}\) per \(r\)-bin is increasing in \(r\) (Spearman \(\rho=0.996\) across 35 bins; rank-tie handling as in the archived script) and saturates at 1 in the top \(r\)-bins (\(r>10\)). The saturation has a structural reading: scores are non-positive, so \(\delta_{\mathrm{rel}}=1\) means \(\operatorname{score}\approx 0\) --- in the measured campaigns, large candidate heights come with scores close to 0, hence far from the strictly negative minimum. A possible analytic route would require at least: (L1b-i) a uniform bound --- every candidate of height \(H\ge C\cdot\mathrm{env}\) has \(\lvert\operatorname{score}\rvert\le f(H)\) with \(f\to 0\), for all admissible states, not only sampled ones; (L1b-ii) a uniform existence statement --- at every admissible state some candidate has \(\operatorname{score}\le -c_K<0\) (the score-side counterpart of the \(C_1=1\) diagnostic above); and, beyond both, a trapping argument confining canonicalized trajectories to a finite state set, with exact tie handling. Only the conjunction would yield periodicity per field; none of the three is proved here. The per-field fitted floor slope is field-dependent (CV = 0.74; no significant normalization by Reg or \(\log B_K\) was found), so the constant in (ii) is a field invariant whose form remains open. Two facts sharpen the target. By the closed form of \S2.2, \(\lvert\operatorname{score}\rvert=\lvert D\rvert^2\cdot\lvert\xi\rvert^2/\lvert\langle\xi\times\nu,\xi\times\nu\rangle\rvert^2\) with \(\lvert D\rvert\) constant along the run, so (i) and (ii) are statements about the ratio \(\lvert\xi\rvert^2/\lvert\langle\xi\times\nu,\xi\times\nu\rangle\rvert^2\): (i) asks that high-height candidates have small ratio (Pl\"ucker norm large relative to \(\lvert\xi\rvert\)), (ii) that some candidate keeps the ratio above a field constant. And the same 973,914-candidate dataset gives the empirical form of \(f\) directly: per-field regression of \(\log\lvert\operatorname{score}\rvert\) on \(\log H\) has slope \(\beta\) with median \(-1.17\) and cross-field CV 0.12. A per-step refinement (120 steps, per-candidate recomputation of both components of the closed form; campaign shipped in the archive) localizes the law completely: at fixed state, \(\log\lvert\xi\rvert^2\) is flat in \(\log H\) (slope \(\approx 0.00\)) while \(\log\lvert\langle\xi\times\nu,\xi\times\nu\rangle\rvert^2\) has slope \(\approx 0.98\), denominator-dominant at all 120 steps --- so the local law is \(\lvert\operatorname{score}\rvert\sim H^{-1}\), the pooled \(-1.17\) being a composition effect across states. The analytic content of (i) thus reduces to a single inequality between unit-invariant quantities: for admissible candidates of large height, \(\lvert\langle\xi\times\nu,\xi\times\nu\rangle\rvert^{2}\ge c_K\,\lvert\xi\rvert^{2}H\) (equivalently \(\lvert\operatorname{score}\rvert\le\lvert D\rvert^{2}/(c_K H)\), the local law above; the ratio is invariant under the unit action, while neither factor separately is).

Figure~\ref{fig:frontier} displays the sampled frontier.

\begin{figure}[t]
\centering
\includegraphics[width=0.88\linewidth]{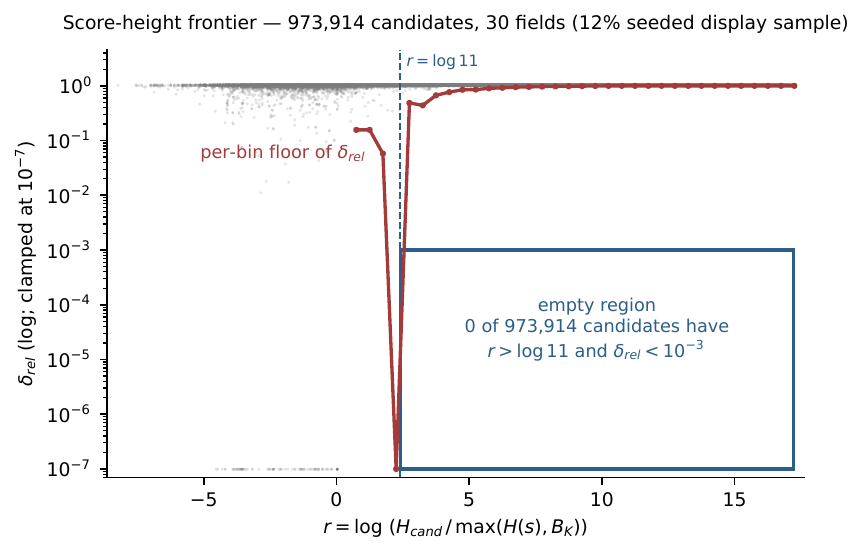}
\caption{The score-height frontier over all 973,914 instrumented candidates (12\% seeded display sample, rasterized; log scale; \(\delta_{\mathrm{rel}}\) clamped at \(10^{-7}\) --- a handful of float-noise negatives documented in the figure script). Red: per-bin floor of \(\delta_{\mathrm{rel}}\) for \(r\ge 0.5\). Framed: the empty region of \S5.3 --- no candidate has \(r>\log 11\) and \(\delta_{\mathrm{rel}}<10^{-3}\). The dense mass at \(\delta_{\mathrm{rel}}\approx 1\) for large \(r\) is the saturation: high candidates have scores near 0.}
\label{fig:frontier}
\end{figure}

\subsection{An honest negative}\label{subsec:negative}

The naive bridge is false: the score-minimal candidate is \emph{not} the height-minimal one. Over the 1016 instrumented steps the selected candidate exceeds the height-minimal candidate by a median factor 2.9, upper quartile 10, maximum \(\times 377\); and the height-minimal candidate has no simple identity (it is the greedy digit on 8\% of steps, \(W\) on 5\%, ``other'' on 87\%). A proof therefore cannot proceed by showing the algorithm picks small heights \emph{pointwise}; what the data supports is the weaker and sufficient frontier statement of \S5.3. This distinction --- measured, not guessed --- is in our view the main structural information this dataset contributes toward a proof.

\section{Machine-verified instances}\label{sec:instances}

\subsection{Method: forward closure and functional graph}\label{subsec:closure-method}

Fix \(K\) and canonicalize states modulo the diagonal action of \(\pm\varepsilon^t\).

\begin{definition}[canonical form]\label{def:canonical-form}
The canonical representative of \(s\) is \(\pm\varepsilon^t\cdot s\) where \(t\in\mathbb{Z}\) is the unique exponent bringing \(\lvert\sigma_r(u_1)\rvert\) into the fundamental window \([1,\lvert\sigma_r(\varepsilon)\rvert^{-1})\) (\(\varepsilon\) chosen contracting in \(\sigma_r\)), and the global sign is fixed by the archived canonicalization key of each table (the shipped conventions differ per table; the functional graph and every statement of this section are therefore statements about states modulo \(\langle\pm\varepsilon^t\rangle\), on which all conventions agree). Since \(\varepsilon\in\mathcal{O}_K^\times\) need not lie in \(\mathbb{Z}[\alpha]\), the canonical representative has exact \emph{rational} coordinates in the basis \((1,\alpha,\alpha^2)\), with denominators dividing the index \([\mathcal{O}_K:\mathbb{Z}[\alpha]]\) (integer coordinates when \(\mathbb{Z}[\alpha]\) is maximal; e.g.~denominators up to 2 occur for \(x^3-4\)). Equality of states is exact equality of these rational coordinates; the dynamics itself runs on integer states throughout. (The exact procedure, including the fixture ``\(s\) and \(\varepsilon\cdot s\) canonicalize identically'', is in the archive.)
\end{definition}

Starting from an explicit seed set, compute the forward closure under one step of the algorithm: apply \(A\), canonicalize the image, add it if new, iterate to stability. On the resulting finite set \(F\), the transition map \(s\mapsto A(s)\) is a functional graph; its cycles and basins are computed exactly. This yields statements of the form:

\begin{theorem}[schema]\label{thm:schema}
The explicit finite set \(F\) is closed under \(A\), its functional graph has a unique cycle --- the canonical cycle of \(K\) --- and every state of \(F\) reaches it.
\end{theorem}

When the transition table itself is certified exactly (as for the plastic instance), such a statement is decided entirely by finite exact computation, with the epistemic status of a computer-verified case check (as in graph-theoretic case analyses), and it settles that \emph{for these explicitly listed states there is nothing left to prove}; when the table is a 160-digit computation (\S6.3), the graph analysis is exact only relative to the fixed table. The general conjecture T*/C* is not touched either way.

\subsection{The plastic instance, exact end to end}\label{subsec:plastic}

\begin{customthm}[plastic instance]{B}\label{thm:B}
For \(K=\mathbb{Q}(\alpha)\), \(\alpha^3=\alpha+1\) (the plastic field, disc \(-23\), \(B_K=1\)). For each nonzero integer \(n\) with \(|n|\le 11\), fix the finite representative list \(R_n\subset K^\times/\langle\pm\varepsilon\rangle\) returned by the archived bnfisintnorm enumeration, ordered by the stored coordinate key; every enumerated element \(u\) of norm \(n\) is written uniquely as \(u=\pm\rho\varepsilon^{e(u)}\) with \(\rho\in R_n\), using the same sign convention as Definition~\ref{def:canonical-form}. Define the plastic enumeration window
\[
\mathcal W_{\mathrm{pl}}=\{s=(u_1,u_2,u_3): (e(u_2)-e(u_1),\,e(u_3)-e(u_1))\in[-7,12]\times[-7,9]\},
\]
after global canonicalization. (The bounds were motivated by the envelope observed across prior campaigns, widened by 3; the theorem concerns this explicitly defined finite enumeration.) The admissible states of height \(H\le 11\) within \(\mathcal W_{\mathrm{pl}}\), with the seed's coordinate determinant \(d_0\), number 2263, computed from the norm equations (bnfisintnorm) for \(\lvert n\rvert\le 11\). Their forward closure stabilizes in 3 iterations, adding 22 states (growth 21, 1, 0): \textbf{the explicit set \(F\) of 2285 canonical states is closed under \(A\), and its functional graph has exactly one cycle --- the canonical one (length 1, digit \(V(1,0,0)\)); all 2285 states flow to it.} This is a certified transition graph on an explicit finite set, containing all in-window admissible states of height \(\le 11\) --- not a statement about all admissible states of that height (\S6.4). The maximum height on \(F\) is 187, reached along transients. Moreover (\S2.3), all 2285 transitions are certified by exact score comparisons in \(\mathbb{Q}(\alpha)\): the instance depends on no floating-point threshold. The residual caveat: identifying ``\(F\supseteq\) all admissible states of height \(\le 11\)'' requires a window lemma (that admissibility confines exponent spreads to \(\mathcal W_{\mathrm{pl}}\)), which is \emph{false} in general --- see \S6.4 --- so the theorem is about the explicit \(F\), plus the separately verified inclusion of the 2263 in-window states.
\end{customthm}

Figure~\ref{fig:depths} displays distances in this graph.

\begin{figure}[t]
\centering
\includegraphics[width=0.62\linewidth]{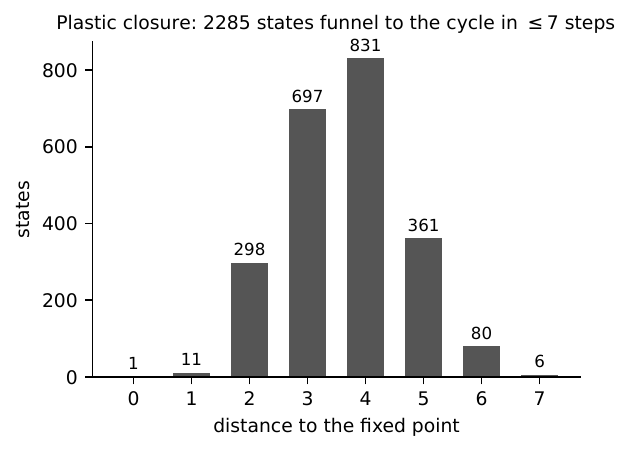}
\caption{Distance to the fixed point in the certified plastic transition graph: the 2285 states funnel to the canonical cycle in \(\le 7\) steps (the certified graph bound).}
\label{fig:depths}
\end{figure}

\subsection{Three further complete instances; twelve orbit-level}\label{subsec:further-instances}

\begin{customprop}[three further complete closures]{C}\label{prop:C}
The same computation for \(x^3-4\) (Karpenkov's published example) closes with \(\lvert F\rvert=69\) and a unique cycle, the canonical one (period 4, the same period length as his published HAPD expansion); its transition table (computed at 160 digits, \S6.3) shows no score ties at that working precision (maximum minimal-score tie class size \(1\)). For \(x^3+x-1\) (disc \(-31\)) --- the field where score ties first appeared --- the closure of the canonical seed together with all deformed states from the Section 4 campaigns has \(\lvert F\rvert=40\), is closed in one iteration, and has the canonical fixed point (digit \(V(0,0,1)\)) as unique cycle, basin 40/40, maximal transient length 7; the 10 tie-transitions inside \(F\) are certified \emph{exact algebraic equalities} of the minimum (by the \(\mathbb{Q}(\alpha)\) method of \S2.3 adapted to this field --- \(J^2\) derived and verified for \(x^3+x-1\); certification table ties\_exact\_x3pxm1.csv in the archive), and the tie-breaking convention reproduces the recorded digits 10/10. For \(x^3-2\) --- the field of \(\sqrt[3]{2}\), one of the Adam--Rhin fields --- the closure has \(\lvert F\rvert=60\), closed in one iteration, unique canonical cycle (period 3), basin 60/60, no score ties observed at 160-digit working precision, and maximal height 8,475,018,297 along a transient from a large deformation seed (pure descent, \S5.2). Thus every complete instance documents its tie status, and every tie-bearing complete instance has its ties decided exactly. Twelve fields in total (the eight smallest of the determinization sample and four quantile-selected by band height --- the three fields above among them) carry the orbit-level version: the forward closure of the canonical orbit alone has a unique cycle in all twelve. Decision precision per complete instance, stated exactly: the plastic transition map is decided exactly end to end (\S2.3); for \(x^3+x-1\) the 10 minimum ties are certified exact and the remaining selections are at 160 digits; the \(x^3-4\) and \(x^3-2\) closure tables are computed at 160 digits (no ties observed). The graph analysis downstream of each table is exact once the table is fixed.
\end{customprop}

\subsection{The admissible set is wilder than the reachable set}\label{subsec:admissible}

It is tempting to read \S6.2 as ``all admissible states of height \(\le 11\) converge''. That statement is \emph{not available}, for a structural reason we document rather than hide. Widening the exponent window by \(\pm 10\) and re-enumerating admissible states of height \(\le 11\) for the plastic field produces 1141 states not in \(F\) (one further enumerated state turned out to be a duplicate of an \(F\)-state and is excluded from all counts). The window did not saturate in the checked range: admissible states exist at the \(+10\) widening tested, so ``admissible with \(H\le 11\)'' is not confined to the certified window (and shows no sign of saturating there). The reachable set is the confined object: none of the 1141 out-window states is ever \emph{visited} by the dynamics started in \(F\). Of these 1141, all 948 within computational reach --- first-step candidate count at most \(10^5\), the out-window flow guard --- converge to the canonical cycle (zero alternative cycles, zero unresolved orbits; maximal height along these flows 463); the remaining 193 are out of reach for an explicit, quantified reason --- their first-step candidate sets are enormous (median enumeration count 6,640 across the 1141, tail up to \(1.67\cdot 10^8\) --- the count itself is computed exactly without enumeration; ``hard'' is defined by that same count exceeding the out-window flow guard, so the count separates the two groups by construction --- the informative fact is the size and shape of its distribution). Unit-exponent spread is the best single structural discriminant of reachable versus out-window admissible states (AUC 0.95), though no exact invariant characterizing reachability was found; we state its identification as an open problem in \S8. Cumulatively, the computed basin of the canonical cycle of the plastic field contains 3,233 distinct canonical states --- the 2285 closure states (certified exactly, \S2.3) and the 948 tested out-window states (computed at 160-digit working precision) --- the two populations disjoint by an exhaustive key intersection; every state flowed in the earlier, partially overlapping campaigns converges as well, and every flow table of the program, for every field, records zero non-canonical terminal cycles.

Figure~\ref{fig:hardness} displays the first-step enumeration costs.

\begin{figure}[t]
\centering
\includegraphics[width=0.66\linewidth]{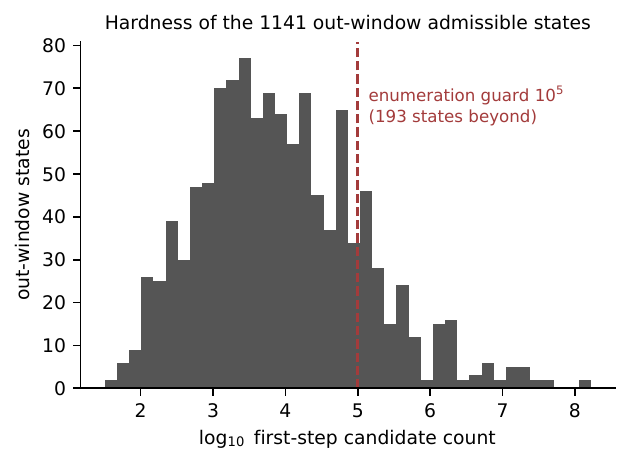}
\caption{First-step enumeration cost of the 1141 out-window admissible states (\(\log_{10}\) candidate count): median 6,640, tail to \(1.67\cdot 10^8\); the out-window flow guard \(10^5\) separates the 948 flowed states from the 193 computationally out of reach.}
\label{fig:hardness}
\end{figure}

\section{Comparison with literal HAPD}\label{sec:hapd}

\subsection{A faithful implementation}\label{subsec:hapd-implementation}

The heuristic algebraic periodicity detecting algorithm (HAPD) is the other algorithm Karpenkov associates with the complex case. We implemented it strictly from its published definition \cite[\S2.3]{Ka22} (state \((\xi,\nu,\mu)\); digit \((a,b)\) minimizing the Markov--Davenport characteristic among admissible \((a,b,1)\); transition matrices as printed), with the published Example~2.11 of \cite{Ka22} as fidelity fixture: our implementation reproduces its expansion of \(x^3-4\) exactly (pre-period 6, period 4).

\subsection{Literal HAPD, as implemented, completes on 50 of 205 polynomials}\label{subsec:hapd-completes}

Run as literally defined (no added conventions, step cap 800) on the 205-polynomial sample, HAPD reaches a certified periodic return on 50/205 polynomials; in the remaining 155 it halts, for two structural reasons: on 80 polynomials the candidate set of the definition becomes \emph{empty} at some step; on 75 the run stops on a tie of the \emph{implemented} integer characteristic --- the compared quantities are exact integers, so these ties are exact for the characteristic as implemented; the published definition specifies no tie-break, so we halt rather than invent a rule, and we make no claim that the implemented characteristic's ties coincide with tie behaviour of the published HAPD definition. The second family is consistent with the exact-tie phenomenon certified for the deterministic \(\sin^2\) score (\S2.3); whether the HAPD ties are the same algebraic degeneracies is not established here, and any minimization-type algorithm without a tie convention is underdetermined wherever exact ties occur. The deterministic \(\sin^2\)-algorithm of this paper, by contrast, completes with certificates on 205/205. We read these halts not as a defect of HAPD --- which was proposed as a heuristic --- but as evidence that determinization is not cosmetic: it is what makes a mass empirical record possible at all. (A determinized HAPD variant might well complete on more inputs; our comparison concerns only the literal published definition.)

\subsection{An observed divisibility pattern of certificate units}\label{subsec:hapd-units}

On the 50 polynomials where both algorithms terminate, the certificate units are never unrelated: writing each certificate as \(\zeta\cdot\varepsilon^k\), we have \(k_{\mathrm{HAPD}}=m\cdot k_{\sin^2}\) with \(m=1\) on 36/50 and \(m\in\{2,3,4\}\) on the remaining 14, \emph{and} \(\zeta_{\mathrm{HAPD}}=\zeta_{\sin^2}^{m}\) on all 50 (torsion checked explicitly) --- hence \(\lambda_{\mathrm{HAPD}}=\lambda_{\sin^2}^{m}\) exactly, so \(\lambda_{\mathrm{HAPD}}\in\langle\lambda_{\sin^2}\rangle\) in all 50 cases, and conversely fails in the 14 cases \(m>1\). The deterministic \(\sin^2\)-algorithm returns at the \emph{minimal} power among the two. Period lengths correlate (Spearman \(\rho=0.634\); descriptive, as all correlations in this paper, \S5) but coincide on only 6/50: the algorithms are genuinely different expansions detecting the same underlying unit group, consistently with \S3.6.

\section{Limitations and open questions}\label{sec:limitations}

\subsection{Status summary}\label{subsec:status}

Table~\ref{tab:status} recapitulates the certification status of every computational block.

\begin{table}[ht]
\centering\footnotesize
\begin{tabular}{@{}p{2.4cm}p{2.2cm}p{3.6cm}p{3.0cm}@{}}
\toprule
Block & Size & Exactness & What it shows\\
\midrule
Mass campaign (\S3.2) & 205 runs, 2265 steps & every transition and tie exact-replayed; unit certificates exact & periodicity of all sign-normalized runs, with units\\
\addlinespace
Box campaign (\S3.3) & 194 runs (97 representatives; 82 distinct digit sequences), 420 distinct steps & same standard as \S3.2 & the same, on an exhaustive mirror-closed box\\
\addlinespace
Deformations (\S4) & 457 starts & runs at 160 digits, re-executed deterministically; cycle identity by exact states; certificates exact & canonical-cycle invariance under change of basis (C*)\\
\addlinespace
Complete closures (\S6) & 4 fields; plastic: 2285 states & plastic transition map fully exact; ties of \(x^3+x-1\) exact; certificates exact; plastic graph machine-checked in Lean 4 (kernel-only) & convergence over explicit finite state sets\\
\addlinespace
Out-window flows (\S6.4) & 948 of 1141 states & 160 digits & reachable \(\subsetneq\) admissible; no alternative cycle\\
\addlinespace
Frontier (\S5.3) & 973{,}914 candidates & instrumented measurement (160 digits) & score--height profile; empty region\\
\addlinespace
HAPD comparison (\S7) & 205 runs & diagnostic & literal HAPD incomplete on (1,1)\\
\bottomrule
\end{tabular}
\caption{Certification status by block.}\label{tab:status}
\end{table}

Established by finite exact computation: the certificates of \S3 (205/205 direct-or-mirror, and 194/194 on the exhaustive box), and, of the \S6 instances, the plastic closure (exact transition map, exactly decided ties, machine-checked graph). Established as 160-digit computations whose graph analysis is exact on the fixed tables: the \(x^3+x-1\) closure (its 10 minimum ties certified exact), and the \(x^3-4\) and \(x^3-2\) closures (\S6.3); the twelve orbit-level closures are of the same numerical status. Measured on samples fixed in advance, with constants but without proof: the descent diagnostics and the score--height frontier of \S5, the canonical-cycle record of \S4 (the cycle-identity comparisons themselves are exact, \S4.1; what lacks proof is the generality). Open: everything general. No statement of this paper resolves Problem 4; the general periodicity theorem (T*) and the canonical-cycle conjecture (C*) remain conjectures.

\subsection{Open problems}\label{subsec:open-problems}

\textbf{OP1 (the analytic lock).} Prove, for complex cubic fields: (i) a candidate state of height \(H\ge C\cdot\max(H(s),B_K)\) has \(\lvert\operatorname{score}\rvert\le f(H)\) with \(f\to 0\); (ii) at every admissible state some candidate has \(\operatorname{score}\le -c_K<0\). Together with proved analogues of the descent diagnostics of \S5.2 and a trapping statement confining canonicalized trajectories to a finite window (OP3), these would give periodicity per field by the finite method of \S6. The data of \S5.3 quantifies both statements; the field-dependence of the constants (floor slope CV = 0.74 across 30 fields; no normalization by Reg or \(\log\lvert\mathrm{disc}\rvert\) found) is part of the problem.

\textbf{OP2 (canonical cycle).} Prove or refute C* (\S4.4). Even for a single field, a proof that \emph{all} admissible states --- not just an explicit finite set --- reach the canonical cycle requires OP1 plus an answer to OP3.

\textbf{OP3 (the reachable set).} Characterize the states actually visited by the dynamics. The admissible set of bounded height is not confined to the certified window (\S6.4; checked at a $+10$ widening), but the reachable set empirically is; unit-exponent spread is the best structural discriminant we found (AUC 0.95), yet no exact invariant. A window lemma (``trajectories started in a finite window \(\mathcal U\) stay in a finite window \(\mathcal U'\)'') would convert the \S6 instance theorems into per-field periodicity theorems.

\textbf{OP4 (the band).} Explain \(B_K\approx\lvert\mathrm{disc}\rvert^{1.37}\) (\(R^2=0.92\), \S5.1) --- presumably an avatar of the covolume/regulator geometry of the unit lattice --- and the rebound constant \(11=C_2\) of \S5.2.

\textbf{OP5 (exact ties).} Characterize the algebraic score ties of \S2.3 (156/2285 states of the plastic closure; all 10 ties of the \(x^3+x-1\) closure, certified exact; and possibly related to the 75 tie-halts of \S7.2, exact for the implemented integer characteristic). Which pairs of unimodular moves are score-equal on which states, and why? A classification would either justify the lexicographic convention structurally or suggest a canonical one.

\textbf{OP6 (the certificate exponent).} Explain the distribution of \(k\) (\S3.4) --- in particular what forces \(k>1\) (observed up to 14) --- beyond the closing-identity constraint. The mechanism is lattice-level, not field-level: the same field carries \(k=1\) and \(k=14\) at different generators (\S4.4).

\textbf{OP7 (formal verification).} The plastic transition map is certified by exact \(\mathbb{Q}(\alpha)\) comparisons; the three other \S6 closure tables are 160-digit computations whose graph analysis is exact once the table is fixed (\S6.3). The combinatorial core of this layer is no longer future work: the full plastic transition graph (2285 states, convergence to the terminal fixed point in at most 7 steps) is machine-checked in Lean 4, kernel-only, with standard axioms and no native decision procedures. Completing a formal instance theorem on top of this sealed graph layer --- the exact score-comparison certificates and the unit identity --- remains a finite, well-scoped project: the integer certificate layer (unimodular matrices, functional graph, unit identity) requires no real analysis at all.

\textbf{OP8 (functoriality of the cycle).} The terminal cycle is not a field invariant and not determined by the polynomial discriminant alone (\S4.4). A ring-isomorphism test across all 33 same-field pairs of the sample sharpens the question: the 3 pairs with isomorphic orders \(\mathbb{Z}[\alpha]\) have equal terminal data, all pairs with equal terminal data have isomorphic orders, and the 5 same-discriminant disagreeing pairs turn out to be non-isomorphic orders. The data is thus consistent, in both directions, with the terminal cycle being an invariant of the isomorphism class of the order \(\mathbb{Z}[\alpha]\) --- three agreeing pairs is thin evidence, but no counterexample exists in the sample. OP8, sharpened: prove or refute that the terminal cycle depends only on the isomorphism class of \((\mathbb{Z}[\alpha],\mathrm{conventions})\).

\subsection{Threats to validity}\label{subsec:threats}

The samples are bounded (the 205-polynomial sample drawn from \([-6,6]^3\), the exhaustive box campaign confined to \([-3,3]^3\)); all constants (11, the floor slope, the exponent 1.37) are sample-calibrated and could drift on wider samples. Step caps (500 for mass and box runs, 800 for literal HAPD, 1000 for ordinary deformation runs, larger in the separate stress scripts) and campaign-specific enumeration guards were never hit in the certified runs, but 4 stress deformations and 193 out-window states remain unresolved for quantified computational reasons --- they are reported, not extrapolated. Deformed starts are seeded products of elementary matrices, a specific ensemble of \(\mathrm{SL}_3(\mathbb{Z})\). After the exact replays, the floating tie threshold \(10^{-60}\) remains relevant only for the auxiliary campaigns that were not replayed exactly (deformations, stress, frontier instrumentation, HAPD comparison); the 205 mass trajectories, the box runs, and the plastic transition map carry no floating-point risk (the tie status of the other closure instances is documented in \S6.3).

\section{Reproducibility}\label{sec:reproducibility}

This research was AI-assisted. All computations were carried out under a fixed audit protocol: questions and thresholds fixed in advance, published fixtures reproduced before every run, independent recounts of every headline number, hash-chained manifests, and adversarial review of the manuscript itself. The protocol is documented in PROTOCOL.md inside the archive; the human author directed the work and takes responsibility for all claims.

The reproducibility archive (the version accompanying this paper is stamped in its VERSION file) distinguishes four verification levels, stated here to avoid ambiguity: (i) \emph{hash checking} --- every shipped table is SHA-256-verified (verify\_all.py, portable, no PARI needed); (ii) \emph{recounting} --- verify\_all.py recomputes the main table-derived headline counts of \S2--\S7 from the shipped tables (certificate distributions, the direct closing identity \(P\cdot\bar g=k\cdot\operatorname{Reg}\), closure sizes and cycles, HAPD divisibility with torsion, out-window counts, band refits); (iii) \emph{regeneration} --- make\_all.py \texttt{-{}-mass} reruns the full \S3 computation (PARI/cypari2, pinned environment: Python 3.11, PARI/GP 2.17.3) from the 205 polynomials as sole input, with output byte-identical to the shipped tables; the heavier \S4--\S6 campaigns ship as reference outputs with their scripts (\texttt{-{}-instances} re-runs them; not required by the default verifier); (iv) \emph{exact certification} --- the plastic transition map, the tie tables, all 2265 transition steps of the 205 mass trajectories and all 840 recorded steps of the 194 box runs, and the state-level cycle-identity comparisons of the 457 deformed runs are decided in exact \(\mathbb{Q}(\alpha)\) arithmetic by dedicated shipped scripts. Quick fixtures run in seconds from a clean extraction.

\section*{Version history}
\emph{v1}: deposited on Zenodo, July 2026 (manuscript \href{https://doi.org/10.5281/zenodo.21222498}{doi:10.5281/zenodo.21222498}). \emph{v2} (this version): editorial revision following external review --- expanded related work (the Vorono\"i--Buchmann unit-algorithm lineage, the HAPD record on complex cubic data, the repetend-matrix certification precedent); the strict-real-ordering lemma (Lemma~\ref{lem:strict-order}) stated and proved explicitly; the HAPD tie-halts restated as exact ties of the implemented integer characteristic, with no fidelity claim to the published characteristic; correlation indicators labeled as exploratory diagnostics; PDF metadata completed. Two additions report work completed after the v1 deposit: the abstract and OP7 now state that the plastic transition graph, with its convergence to the terminal cycle, is machine-checked in Lean 4 (kernel-only, standard axioms) --- in v1 this was described as a well-scoped project --- and the abstract's box-campaign sentence now names the 97 positive-root representative expansions (82 distinct digit sequences). Also added after v1: Conjecture T* now includes indefinite well-definedness, with a remark recording that this clause is settled by Lemma~\ref{lem:strict-order} and the companion theory paper's non-isotropy theorem; Theorem~B's enumeration is restated via the representative lists \(R_n\), the exponent map \(e(u)\) and the window \(\mathcal W_{\mathrm{pl}}\); the tie status of the \(x^3-4\) closure is recorded in Proposition~C; and the theory companion is now cited. Statement corrections: the closure-instance exactness statuses are now uniform across the abstract, \S2.2, \S2.5, Proposition~C and \S8.1 (mass, box and the plastic closure exact; the three other closure tables are 160-digit computations with exact graph analysis, \S6.3); a rotation-diagnostic $p$-value corrected to its recomputed value ($0.38$); the HAPD subsection title now says ``as implemented''; and the pre-period indexing convention of \S2.4 is restated to match the recorded values exactly (zero-based state indexing from the seed \(s_0\), first-recurrence index clamped below at 1); the v1 one-based phrasing did not match the recorded values in runs with a transient prefix. The recorded values themselves are unchanged, and every certificate identity holds at the recorded index. Beyond the statement corrections listed above, no principal result or computed datum of v1 is modified.

\section*{Acknowledgments}
The author thanks Karsten M\"uller for helpful correspondence.

\section*{Data availability}
A reproducibility archive containing all scripts, data and verification tools accompanies this paper: \href{https://doi.org/10.5281/zenodo.21182759}{doi:10.5281/zenodo.21182759} (v9, including the self-contained Lean formalization layer with replay instructions; archive SHA-256: \texttt{72808bae\allowbreak{}53ff2eb9\allowbreak{}36ce3360\allowbreak{}4d466c9b\allowbreak{}ef2968f2\allowbreak{}4b4dc7c2\allowbreak{}13279ad3\allowbreak{}90ab7cb2}).

\appendix

\section{The 205 input polynomials}\label{app:sample}
Monic cubics \(x^3+a_2x^2+a_1x+a_0\), listed as triples \((a_2,a_1,a_0)\): the first 200 in the
lexicographic enumeration of \S3.1, followed by the five literature-tagged polynomials. The
machine-readable list (SHA-256 \texttt{2cf13f23\allowbreak{}7abba972\allowbreak{}\dots}, full
hash in the archive manifest) is the sole data input of the regeneration target.

\begingroup\scriptsize\setlength{\tabcolsep}{4pt}
\begin{center}
\begin{tabular}{lllll}
\((-6,-6,-6)\) & \((-6,-6,-5)\) & \((-6,-6,-4)\) & \((-6,-6,-3)\) & \((-6,-6,-2)\) \\
\((-6,-5,-6)\) & \((-6,-5,-5)\) & \((-6,-5,-4)\) & \((-6,-5,-3)\) & \((-6,-5,-2)\) \\
\((-6,-5,-1)\) & \((-6,-4,-6)\) & \((-6,-4,-5)\) & \((-6,-4,-4)\) & \((-6,-4,-3)\) \\
\((-6,-4,-2)\) & \((-6,-4,-1)\) & \((-6,-3,-6)\) & \((-6,-3,-5)\) & \((-6,-3,-4)\) \\
\((-6,-3,-3)\) & \((-6,-3,-2)\) & \((-6,-3,-1)\) & \((-6,-2,-6)\) & \((-6,-2,-5)\) \\
\((-6,-2,-4)\) & \((-6,-2,-3)\) & \((-6,-2,-2)\) & \((-6,-2,-1)\) & \((-6,-1,-6)\) \\
\((-6,-1,-5)\) & \((-6,-1,-4)\) & \((-6,-1,-3)\) & \((-6,-1,-2)\) & \((-6,-1,-1)\) \\
\((-6,0,-6)\) & \((-6,0,-5)\) & \((-6,0,-4)\) & \((-6,0,-3)\) & \((-6,0,-2)\) \\
\((-6,0,-1)\) & \((-6,1,-5)\) & \((-6,1,-4)\) & \((-6,1,-3)\) & \((-6,1,-2)\) \\
\((-6,1,-1)\) & \((-6,2,-6)\) & \((-6,2,-5)\) & \((-6,2,-4)\) & \((-6,2,-3)\) \\
\((-6,2,-2)\) & \((-6,2,-1)\) & \((-6,3,-6)\) & \((-6,3,-5)\) & \((-6,3,-4)\) \\
\((-6,3,-3)\) & \((-6,3,-2)\) & \((-6,3,-1)\) & \((-6,4,-6)\) & \((-6,4,-5)\) \\
\((-6,4,-4)\) & \((-6,4,-3)\) & \((-6,4,-2)\) & \((-6,4,-1)\) & \((-6,5,-6)\) \\
\((-6,5,-5)\) & \((-6,5,-4)\) & \((-6,5,-3)\) & \((-6,5,-2)\) & \((-6,6,-6)\) \\
\((-6,6,-4)\) & \((-6,6,-3)\) & \((-6,6,-2)\) & \((-5,-6,-6)\) & \((-5,-6,-5)\) \\
\((-5,-6,-4)\) & \((-5,-6,-3)\) & \((-5,-6,-2)\) & \((-5,-5,-5)\) & \((-5,-5,-4)\) \\
\((-5,-5,-3)\) & \((-5,-5,-2)\) & \((-5,-4,-6)\) & \((-5,-4,-5)\) & \((-5,-4,-4)\) \\
\((-5,-4,-3)\) & \((-5,-4,-2)\) & \((-5,-4,-1)\) & \((-5,-3,-6)\) & \((-5,-3,-5)\) \\
\((-5,-3,-4)\) & \((-5,-3,-3)\) & \((-5,-3,-2)\) & \((-5,-3,-1)\) & \((-5,-2,-6)\) \\
\((-5,-2,-5)\) & \((-5,-2,-4)\) & \((-5,-2,-3)\) & \((-5,-2,-2)\) & \((-5,-2,-1)\) \\
\((-5,-1,-6)\) & \((-5,-1,-5)\) & \((-5,-1,-4)\) & \((-5,-1,-3)\) & \((-5,-1,-2)\) \\
\((-5,-1,-1)\) & \((-5,0,-6)\) & \((-5,0,-5)\) & \((-5,0,-4)\) & \((-5,0,-3)\) \\
\((-5,0,-2)\) & \((-5,0,-1)\) & \((-5,1,-6)\) & \((-5,1,-4)\) & \((-5,1,-3)\) \\
\((-5,1,-2)\) & \((-5,1,-1)\) & \((-5,2,-6)\) & \((-5,2,-5)\) & \((-5,2,-4)\) \\
\((-5,2,-3)\) & \((-5,2,-2)\) & \((-5,2,-1)\) & \((-5,3,-6)\) & \((-5,3,-5)\) \\
\((-5,3,-4)\) & \((-5,3,-3)\) & \((-5,3,-2)\) & \((-5,3,-1)\) & \((-5,4,-6)\) \\
\((-5,4,-5)\) & \((-5,4,-4)\) & \((-5,4,-3)\) & \((-5,4,-2)\) & \((-5,4,-1)\) \\
\((-5,5,-6)\) & \((-5,5,-5)\) & \((-5,5,-3)\) & \((-5,5,-2)\) & \((-5,5,4)\) \\
\((-5,5,5)\) & \((-5,5,6)\) & \((-5,6,-6)\) & \((-5,6,-5)\) & \((-5,6,-4)\) \\
\((-5,6,-3)\) & \((-5,6,1)\) & \((-5,6,2)\) & \((-5,6,3)\) & \((-5,6,4)\) \\
\((-5,6,5)\) & \((-5,6,6)\) & \((-4,-6,-6)\) & \((-4,-6,-5)\) & \((-4,-6,-4)\) \\
\((-4,-6,-3)\) & \((-4,-6,-2)\) & \((-4,-5,-6)\) & \((-4,-5,-5)\) & \((-4,-5,-4)\) \\
\((-4,-5,-3)\) & \((-4,-5,-2)\) & \((-4,-4,-6)\) & \((-4,-4,-4)\) & \((-4,-4,-3)\) \\
\((-4,-4,-2)\) & \((-4,-4,-1)\) & \((-4,-3,-6)\) & \((-4,-3,-5)\) & \((-4,-3,-4)\) \\
\((-4,-3,-3)\) & \((-4,-3,-2)\) & \((-4,-3,-1)\) & \((-4,-2,-6)\) & \((-4,-2,-5)\) \\
\((-4,-2,-4)\) & \((-4,-2,-3)\) & \((-4,-2,-2)\) & \((-4,-2,-1)\) & \((-4,-1,-6)\) \\
\((-4,-1,-5)\) & \((-4,-1,-4)\) & \((-4,-1,-3)\) & \((-4,-1,-2)\) & \((-4,-1,-1)\) \\
\((-4,0,-6)\) & \((-4,0,-5)\) & \((-4,0,-4)\) & \((-4,0,-3)\) & \((-4,0,-2)\) \\
\((-4,0,-1)\) & \((-4,1,-6)\) & \((-4,1,-5)\) & \((-4,1,-3)\) & \((-4,1,-2)\) \\
\((-4,1,-1)\) & \((-4,2,-6)\) & \((-4,2,-5)\) & \((-4,2,-4)\) & \((-4,2,-3)\) \\
\((0,0,-2)\) & \((0,0,-17)\) & \((-3,0,-2)\) & \((2,1,4)\) & \((0,0,-4)\) \\
\end{tabular}
\end{center}
\endgroup

\section{One step by hand; certificate format, with two worked examples}\label{app:certificates}

\emph{One full step, by hand.} Take the plastic field \(p=x^3-x-1\) (\(\alpha\approx 1.3247\)) and the canonical seed \(s_0=(1,\alpha,\alpha^2)\). Step 1 of \S2.5: sorting by the real embedding gives \((u_1,u_2,u_3)=(\alpha^2,\alpha,1)\), \((x,y,z)\approx(1.7549,1.3247,1)\), all positive. Step 2: the ranges of \S2.1 give \(0\le a\le\lfloor x/z\rfloor=1\), \(0\le b\le\lfloor y/z\rfloor=1\), \(g\) bounded accordingly, and \(W\) is admissible since \(z>x-y>0\) (here \(1>0.4302>0\)). Twelve candidates have positive embeddings; none is isotropic. Step 3, their scores by the closed form of \S2.2 (six decimals shown; all comparisons are decided exactly):

\begin{center}\footnotesize
\begin{tabular}{@{}llr@{}}
\toprule
digit & new state (coordinate rows) & score\\
\midrule
\(V(1,0,0)\) & \((-1,0,1),(0,1,0),(1,0,0)\) & \(-3.504475\)\\
\(V(1,1,1)\) & \((0,-1,1),(-1,1,0),(1,0,0)\) & \(-1.127361\)\\
\(W\) & \((0,-1,1),(0,1,0),(1,1,-1)\) & \(-1.127361\)\\
\(V(0,0,1)\) & \((0,-1,1),(0,1,0),(1,0,0)\) & \(-0.960344\)\\
\(V(1,1,0)\) & \((-1,0,1),(-1,1,0),(1,0,0)\) & \(-0.960344\)\\
\(V(0,1,0)\) & \((0,0,1),(-1,1,0),(1,0,0)\) & \(-0.440812\)\\
\(V(0,1,1)\) & \((1,-1,1),(-1,1,0),(1,0,0)\) & \(-0.371747\)\\
\(V(1,1,2)\) & \((1,-2,1),(-1,1,0),(1,0,0)\) & \(-0.136377\)\\
\(V(0,1,2)\) & \((2,-2,1),(-1,1,0),(1,0,0)\) & \(-0.107367\)\\
\(V(0,1,3)\) & \((3,-3,1),(-1,1,0),(1,0,0)\) & \(-0.026386\)\\
\(V(0,1,4)\) & \((4,-4,1),(-1,1,0),(1,0,0)\) & \(-0.007578\)\\
\(V(0,1,5)\) & \((5,-5,1),(-1,1,0),(1,0,0)\) & \(-0.002803\)\\
\bottomrule
\end{tabular}
\end{center}

Step 4: the minimum is \(V(1,0,0)\), separated from the runner-up by a factor of about~3 --- no tie at this step. Step 5: the digit is \(V(1,0,0)\), the new state is \((\alpha^2-1,\alpha,1)\), and from there the run repeats \(V(1,0,0)\): \(m=1\), \(P=1\), closed by the first certificate below. Note the two score coincidences away from the minimum (\(V(1,1,1)\)/\(W\) and \(V(0,0,1)\)/\(V(1,1,0)\), equal to 120 digits): the two displayed coincidences are away from the minimum and were checked numerically to 120 digits; exact algebraic equalities \emph{of the minimum} are certified at 156 of the 2285 plastic-graph states (\S2.3) --- score coincidences are thus visible already at the very first step of the very smallest field.

\medskip
Each run's certificate consists of: the pre-period \(m\) and period \(P\); the scalar
\(\lambda\in K\) with \(s_{m+P}=\lambda\cdot s_m\) componentwise in exact field arithmetic; and
the four checks of \S2.4 (same \(\lambda\) on all three components; \(|N_{K/\mathbb{Q}}(\lambda)|=1\);
\(\mathbb{Q}(\lambda)=K\); \(\lambda\in\mathcal{O}_K^\times\) via its \(\zeta\cdot\varepsilon^k\)
decomposition). Two examples:

\emph{Plastic field} (\(p=x^3-x-1\), determinization campaign): \(m=1\), \(P=1\),
\(\lambda=\alpha^2-1\), \(N(\lambda)=1\), \(\zeta=+1\), \(k=1\); the same \(\lambda\) closes all
three components, and \(\alpha^2-1=1/\alpha\) is the inverse of the fundamental unit.

\emph{First mass polynomial} (\(p=x^3-6x^2-6x-6\), disc \(-7884\)): \(m=1\), \(P=3\),
\(\lambda=-\alpha+7\), \(\zeta=+1\), \(k=1\), \(N(\lambda)=1\).

In the archive, each mass-table row carries these fields explicitly, and each transition of the
exact replays carries the winning digit, the runner-up, and a certified separating gap (or an
exact-tie flag with the tie set).

\end{document}